\documentclass[12pt]{article}
\usepackage{amssymb, amsmath}
\newtheorem{prp}{Proposition}[section]

\newtheorem{thm}{Theorem}[section]

\newenvironment{proof}{{\noindent\it Proof. }}{\hfill$\Box$\\}

\newcommand{\Abar}{{\backslash\kern-7pt A}}
\title{How to specify an approximate numerical result}
\author{Nicolas \textsc{Bouleau}
          \footnote{Ecole Nationale des Ponts et Chauss\'ees, ParisTech, Paris 75007 France.}}
\author{Nicolas Bouleau}         
\date{---}

\begin{document}
%

\maketitle
\noindent classification :{ 60Hxx, 31C25, 94B70, 65Gxx.}

\begin{abstract}     
The Dirichlet forms methods, in order to represent errors and their propagation, are particularly powerful in infinite dimensional problems such as models involving stochastic analysis encountered in finance or physics, cf. \cite{NBlivre}. Now, coming back to the finite dimensional case, these methods give a new light on the very classical concept of `numerical approximation' and suggest changes in the habits. We show that for some kinds of approximations only an Ito-like second order differential calculus is relevant to describe and propagate numerical errors through a mathematical model. We call these situations {\it strongly stochastic}. The main point of this work is an argument based on the {\it arbitrary functions principle} of Poincar\'e-Hopf showing that the errors due to measurements with graduated instruments are strongly stochastic. Eventually we discuss the consequences of this phenomenon on the specification of an approximate numerical result.
\end{abstract}

\section{The dichotomy of small errors.}

Let us begin by showing that there are two kinds of small errors which do not propagate according to the same differential calculus.

Suppose two applied mathematicians A and B attempt to perform stochastic simulation rigourously. By means of the well known inversion and rejection methods, they are able to simulate any probability law provided that they can pick up a real number in the unit interval $[0,1]$ randomly.

For this, the researcher A chooses the method of drawing the binary digits by heads or tails. The researcher B, instead, prefers using a Polya's urn. 

Let us compare the biases and the variances of the errors in the two procedures.

\underline{In the case A}, the real number 
$$x=0,a_1a_2a_3\cdots =\sum_{k=1}^\infty \frac{a_k}{2^k}\quad\qquad a_k\in\{0,1\}$$
is approximated by $x_n=\sum_{k=1}^n\frac{a_k}{2^k}.$

 Denoting $\mathcal{F}_n$ the $\sigma$-field generated by $a_1, \cdots,a_n$, the bias of the error is 
$$b_n=\mathbb{E}[(x-x_n)|\mathcal{F}_n]=\sum_{k=n+1}^\infty\frac{1/2}{2^k}=\frac{1}{2^{n+1}}$$
and the variance of the error is 
$$v_n=\mathbb{E}[(x-x_n)^2|\mathcal{F}_n]-(\mathbb{E}[(x-x_n)|\mathcal{F}_n])^2=\frac{1}{3}\frac{1}{4^n}-\frac{1}{4}\frac{1}{4^n}=\frac{1}{12}\frac{1}{4^n}.$$

For \underline{the case B}, let us recall the principle of Polya's urn : there is at the beginning a white ball and a black ball in the urn and each time a ball is drawn from the urn, it is put back into the urn together with an other ball of the same colour. 

After $n$ drawings, the proportion $X_n$ of white balls in the urn is given by $$(n+2)X_{n}=(n+1)X_{n-1}+1_{\{U_{n}\leq X_{n-1}\}}$$
where $U_n$ is uniformly distributed on $[0,1]$ independent of $\mathcal{F}_{n-1}=\sigma(X_0,\ldots, X_{n-1})$. In other words
$$X_n=X_{n-1}+\frac{1}{n+2}(1_{\{U_{n}\leq X_{n-1}\}}-X_{n-1}).$$
$X_n$ is a bounded martingale which converges a.s. and in $L^p$, $p\in[1,\infty[$, to a random variable $X_\infty$ uniformly distributed on $[0,1]$ as easily seen when the initial configuration of the urn is one white ball and one black ball.

For the bias we have
$$b_n=\mathbb{E}[X_\infty-X_n|\mathcal{F}_n]=0$$ and for the variance
$$v_n=\mathbb{E}[(X_\infty-X_n)^2|\mathcal{F}_n]$$ we have $\mathbb{E}[v_n]=\frac{1}{6n}+o(1/n).$

We see that in case A the variances are smaller than the biases, while in case B the biases are smaller than the variances.

How will these errors propagate through the simulations of our two Monte Carlo practioners ?

A Taylor expansion on a $\mathcal{C}^3$-function with bounded derivatives gives 
$$
\begin{array}{rl}
f(X)-f(X_n)=(X-X_n)f^\prime(X_n)&+\frac{1}{2}(X-X_n)^2f^{\prime\prime}(X_n)\\
&+\frac{1}{6}(X-X_n)^3f^{\prime\prime\prime}(X_n+\theta(X-X_n))
\end{array}$$
$$\mbox{ new bias }=\mathbb{E}[f(X_\infty)-f(X_n)|\mathcal{F}_n]=
b_nf^\prime(X_n)+\frac{1}{2}(v_n+b_n^2)f^{\prime\prime}(X_n)+o(v_n)$$
$$
\begin{array}{rl}
\mbox{ new variance }&= \mathbb{E}[(f(X_\infty)-f(X_n))^2|\mathcal{F}_n]-(\mathbb{E}[f(X_\infty)-f(X_n)|\mathcal{F}_n])^2\\
&=(v_n-b_n^2)f^{\prime 2}(X_n)+o(v_n).
\end{array}$$
We can distinguish three cases

1) If the variance is negligible with respect to the bias, $v_n\ll b_n$, (case of researcher A), the dominant term for the bias is asymptotically the first one. $\mathbb{E}[(f(X_\infty)-f(X_n))^2|\mathcal{F}_n]$ is negligible with respect to $\mathbb{E}[f(X_\infty)-f(X_n)|\mathcal{F}_n]$ and the situation will be carried on. It is enough to use the formula
\begin{equation}
\mathbb{E}[f(X_\infty)-f(X_n)|\mathcal{F}_n]=b_nf^\prime(X_n)+o(b_n)
\end{equation}

2) If the variance is of the same order of magnitude as the bias, the situation will persist. The propagation formulae are
\begin{equation}
\left.
\begin{array}{rl}
\mathbb{E}[f(X_\infty)-f(X_n)|\mathcal{F}_n]&=b_nf^\prime(X_n)+\frac{1}{2}v_nf^{\prime\prime}(X_n)+o(v_n)\\
\mathbb{E}[(f(X_\infty)-f(X_n))^2|\mathcal{F}_n]&=v_nf^{\prime 2}(X_n)+o(v_n)
\end{array}
\right\}
\end{equation}

3) If the bias is negligible in comparison to the variance, $b_n\ll v_n$, (case of researcher B), the main term in the bias becomes $\frac{1}{2}v_nf^{\prime\prime}(X_n)$ and we fall in the case 2) where biases and variances remain of the same order of magnitude.\\

We see that a first order differential calculus is relevant for the researcher A.
But instead, the researcher B (with Polya's urn) must perform an error calculus involving both biases and variances, and 

- the error calculus on the variances is a first order differential calculus,

- the error calculus on the biases is a second order differential calculus and uses the calculus on variances.\\

The first case will be called the {\it weakly stochastic} case. Then the usual differential calculus is enough to propagate errors and to assess the sensitivity of the model outputs to data.

The second case will be called {\it strongly stochastic}. Then the propagation of biases (which is important in non-linear models) needs an Ito-like differential calculus given by formulae (1.2).\\

\noindent{\bf Comment.} In practice, generally, we do not control the nature of the errors. In modelling, errors on the data are {\it exogenous}, we know few from where they come. It is wise to think according to the second case, especially to take in account the randomness of the errors through the non-linearities of the model.\\

Let us go deeper into the mathematical arguments by displaying the bias operators and the variance operator (the Dirichlet form) associated with an approximation.

\section{The bias operators and the Dirichlet form
 associated with an approximation.}
 When two random variables $Y$ and $Y_n$ are close together, the asymptotic behaviour of 
 $$\mathbb{E}[(\phi(Y_n)-\phi(Y))\chi(Y)]$$
 and of 
 $$\mathbb{E}[(\phi(Y_n)-\phi(Y))\chi(Y_n)]$$
 where $\phi$ and $\chi$ are test functions, are generally different. As a consequence several bias operators have to be distinguished (cf. \cite{JFA}) :
 
 Let $Y$ be a random variable  defined on $(\Omega, \mathcal{A}, \mathbb{P})$ with values in a measurable space $(E,\mathcal{F})$ and let $Y_n$ be approximations also defined on $(\Omega, \mathcal{A}, \mathbb{P})$ with values in $(E,\mathcal{F})$. We consider an algebra $\mathcal{D}$ of bounded functions from $E$ into $\mathbb{R}$ or $\mathbb{C}$ containing the constants and dense in $L^2(E,\mathcal{F},\mathbb{P}_Y)$ and a sequence $\alpha_n$ of positive numbers. With $\mathcal{D}$ and $(\alpha_n)$ we consider the four following assumptions defining the four bias operators 
$$
(\mbox{H}1)\qquad\left\{\begin{array}{l}
\forall \varphi\in{\cal D}, \mbox{ there exists } \overline{A}[\varphi]\in L^2(E,{\cal F},\mathbb{P}_Y)\quad s.t. \quad\forall \chi\in{\cal D}\\
\lim_{n\rightarrow\infty} \alpha_n\mathbb{E}[(\varphi(Y_n)-\varphi(Y))\chi(Y)]=\mathbb{E}_Y[\overline{A}[\varphi]\chi].
\end{array}\right.
$$
$$
(\mbox{H}2)\qquad\left\{\begin{array}{l}
\forall \varphi\in{\cal D}, \mbox{ there exists } \underline{A}[\varphi]\in L^2(E,{\cal F},\mathbb{P}_Y)\quad s.t. \quad\forall \chi\in{\cal D}\\
\lim_{n\rightarrow\infty} \alpha_n\mathbb{E}[(\varphi(Y)-\varphi(Y_n))\chi(Y_n)]=\mathbb{E}_Y[\underline{A}[\varphi]\chi].
\end{array}\right.
$$
$$
(\mbox{H}3)\quad\left\{\begin{array}{l}
\forall \varphi\in{\cal D}, \mbox{ there exists } \widetilde{A}[\varphi]\in L^2(E,{\cal F},\mathbb{P}_Y)\quad s.t. \quad\forall \chi\in{\cal D}\\
\lim_{n\rightarrow\infty} \alpha_n\mathbb{E}[(\varphi(Y_n)-\varphi(Y))(\chi(Y_n)-\chi(Y))]=-2\mathbb{E}_Y[\widetilde{A}[\varphi]\chi].
\end{array}\right.
$$
$$
(\mbox{H}4)\quad\left\{\begin{array}{l}
\forall \varphi\in{\cal{D}}, \mbox{ there exists } \Abar[\varphi]\in L^2(E,{\mathcal{F}},\mathbb{P}_Y)\quad s.t. \quad\forall \chi\in{\cal{D}}\\
\lim_{n\rightarrow\infty} \alpha_n\mathbb{E}[(\varphi(Y_n)-\varphi(Y))(\chi(Y_n)+\chi(Y))]=2\mathbb{E}_Y[\Abar[\varphi]\chi].
\end{array}\right.
$$
We first note that as soon as two of hypotheses (H1) (H2) (H3) (H4) are fulfilled (with
 the same algebra ${\cal D}$ and the same sequence $\alpha_n$), the other two follow thanks to the relations
$$\widetilde{A}=\frac{\overline{A}+\underline{A}}{2}\quad\quad\Abar=\frac{\overline{A}-\underline{A}}{2}.$$
When defined, the operator $\overline{A}$ which considers the asymptotic error from the point of view of the limit model, will be called {\it the 
theoretical bias operator}.

The operator $\underline{A}$ which considers the asymptotic error from the point of view of the approximating model will be called {\it the 
practical bias operator}.

Because of the property
$$<\widetilde{A}[\varphi],\chi>_{L^2(\mathbb{P}_Y)}=<\varphi,\widetilde{A}[\chi]>_{L^2(\mathbb{P}_Y)}$$
the operator $\widetilde{A}$ will be called {\it the symmetric bias operator}.

The operator $\Abar$ which is often (see theorem 2.2 below) a first order operator will be called {\it the singular bias operator}.\\

\begin{thm} {\it Under the hypothesis {\rm (H3)},

a) the limit
\begin{equation}\widetilde{\cal E}[\varphi,\chi]=\lim_n  \frac{\alpha_n}{2}\mathbb{E}[(\varphi(Y_n)-\varphi(Y))(\chi(Y_n)-\chi(Y)]\qquad \varphi, \chi\in{\cal D}\end{equation}
defines a closable positive bilinear form whose smallest closed extension is denoted $({\cal E},\mathbb{D})$.

b) $({\cal E},\mathbb{D})$ is a Dirichlet form {\rm(cf. \cite{Fuku})}

c) $({\cal E},\mathbb{D})$ admits a square field operator $\Gamma$ satisfying $\forall \varphi,\chi\in{\cal D}$
\begin{equation}
\Gamma[\varphi]=\widetilde{A}[\varphi^2]-2\varphi\widetilde{A}[\varphi]\end{equation}
\begin{equation}\mathbb{E}_Y[\Gamma[\varphi]\chi]=\lim_n\alpha_n\mathbb{E}[(\varphi(Y_n)-\varphi(Y))^2(\chi(Y_n)+\chi(Y))/2]\end{equation}
\indent d) $({\cal E},\mathbb{D})$ is local if and only if $\forall \varphi\in{\cal D}$
\begin{equation}\lim_n \alpha_n\mathbb{E}[(\varphi(Y_n)-\varphi(Y))^4]=0\end{equation}
this condition is equivalent to
$\quad\exists\lambda>2\quad\lim_n\alpha_n\mathbb{E}[|\varphi(Y_n)-\varphi(Y)|^\lambda]=0.$

e) If the form $({\cal E},\mathbb{D})$  is local, then the} principle of asymptotic error calculus
{\it is valid on 
$\widetilde{\cal D}=\{F(f_1,\ldots,f_p)\;:\;f_i\in{\cal D},\;\;F\in{\cal C}^1(\mathbb{R}^p,\mathbb{R})\}$
i.e.\\

$
\lim_n\alpha_n\mathbb{E}[(F(f_1(Y_n),\ldots,f_p(Y_n))-F(f_1(Y),\ldots,f_p(Y)))^2]$\hfill

\hfill$=\mathbb{E}_Y[\sum_{i,j=1}^p F^\prime_i(f_1,\ldots,f_p)F^\prime_j(f_1,\ldots,f_p)\Gamma[f_i,f_j]].$}
\end{thm}
The proof of this theorem is given in [6] Theorem 1, Remark 3 and Theorem 2.
The point e) of the theorem is a commutativity of limits, it means that the error on a function of $Y$ may be directly obtained starting from the error on $Y$ by functional calculus.

An operator $B$ from ${\cal D}$ into $L^2(\mathbb{P}_Y)$ will be said to be a {\it first order operator} if it satisfies
$$B[\varphi\chi]=B[\varphi]\chi+\varphi B[\chi]\qquad\forall\varphi,\chi\in{\cal D}$$
\vspace{-.1cm}

\begin{thm} {\it Under} {\rm(H1)} {\it to} {\rm(H4)} {\it 

a) the}  theoretical variance $\lim_n\alpha_n\mathbb{E}[(\varphi(Y_n)-\varphi(Y))^2\psi(Y)]$ {\it and  the } practical variance
$\lim_n\alpha_n\mathbb{E}[(\varphi(Y_n)-\varphi(Y))^2\psi(Y_n)]$ {\it exist and we have $\forall\varphi,\chi,\psi\in{\cal D}$
$$
\begin{array}{c}
\lim_n\alpha_n\mathbb{E}[(\varphi(Y_n)-\varphi(Y))(\chi(Y_n)-\chi(Y))\psi(Y)]\qquad\qquad\qquad\qquad\\
\qquad\qquad\qquad\qquad=\mathbb{E}_Y[-\underline{A}[\varphi\psi]\chi+\underline{A}[\psi]\varphi\chi
-\overline{A}[\varphi]\chi\psi]\\
\lim_n\alpha_n\mathbb{E}[(\varphi(Y_n)-\varphi(Y))(\chi(Y_n)-\chi(Y))\psi(Y_n)]\qquad\qquad\qquad\qquad\\
\qquad\qquad\qquad\qquad=\mathbb{E}_Y[-\overline{A}[\varphi\psi]\chi+\overline{A}[\psi]\varphi\chi
-\underline{A}[\varphi]\chi\psi]
\end{array}
$$
\indent \it b) These two variances coincide if and only if $\;\Abar$ is a first order operator, and then are equal to 
$\mathbb{E}_Y[\Gamma[\varphi]\psi].$}
 \end{thm}
 The proof of this result is given in [6] Proposition 2.
 \vspace{.3cm}
 
\noindent{\bf Example: Typical formulae of finite dimensional error calculus.}\\

Let us consider a triplet of real random variables $(Y,Z,T)$ and a real random variable $G$ independent of $(Y,Z,T)$ centered with variance one.
We are interested in the approximation $Y_\varepsilon$ of $Y$ given by
\begin{equation}
Y_\varepsilon=Y+\varepsilon Z+\sqrt{\varepsilon}TG.
\end{equation}
In the multidimensional case, $Y$ is with values in $\mathbb{R}^p$ as $Z$, $T$ is a $p\!\times\!q$-matrix and $G$ is independent of $(Y,Z,T)$ with 
values in $\mathbb{R}^q$, centered, square integrable, such that $\mathbb{E}[G_iG_j]=\delta_{ij}.$\\

\noindent{\bf Operator $\overline{A}$.}

\begin{prp} {\it If $Z$ and $T$ are square integrable, if $\varphi$ is ${\cal C}^2$ bounded with bounded derivatives of first and
second orders ($\varphi\in{\cal C}^2_b$) and if $\chi$ is bounded,
$$\frac{1}{\varepsilon}\mathbb{E}[(\varphi(Y_\varepsilon)-\varphi(Y))\chi(Y)]\rightarrow\mathbb{E}_Y[\overline{A}[\varphi]\chi]$$
where $\overline{A}[\varphi](y)=\mathbb{E}[Z|Y\!=\!y]\varphi^\prime(y)+\frac{1}{2}\mathbb{E}[T^2|Y\!=\!y]\varphi^{\prime\prime}(y)$. 

In the multidimensional case 
$$\overline{A}[\varphi](y)=\mathbb{E}[Z^t|Y\!=\!y]\nabla\varphi(y)+\frac{1}{2}\sum_{ij}\mathbb{E}[(TT^t)_{ij}|Y\!=
\!y]\varphi^{\prime\prime}_{ij}(y).$$}
\end{prp}
\begin{proof} Let us give the argument with the notation of the case $q=p=1$. The Taylor-Lagrange formula applied up to second order gives
$$
\begin{array}{c}\frac{1}{\varepsilon}\mathbb{E}[(\varphi(Y_\varepsilon)-\varphi(Y))\chi(Y)]=\mathbb{E}[Z\varphi^\prime(Y)\chi(Y)]\qquad
\qquad\hspace{5.5cm}\\
\qquad\qquad\quad+\frac{1}{2}
\mathbb{E}[(\varepsilon Z^2+2\sqrt{\varepsilon}ZTG+T^2G^2)\qquad\\
\qquad\qquad\qquad\qquad\qquad\int_0^1\int_0^1\varphi^{\prime\prime}(Y+ab(\varepsilon Z+
\sqrt{\varepsilon}TG))2adadb\,\chi(Y)]
\end{array}
$$
(note that $ZTG$ and $T^2G^2\in L^1$ because of the independence) and  this tends by dominated Lebesgue theorem
to
$\mathbb{E}[Z\varphi^\prime(Y)\chi(Y)]+\frac{1}{2}\mathbb{E}[T^2\varphi^{\prime\prime}(Y)\chi(Y)].$
\end{proof}

\noindent{\bf Quadratic form and operator $\widetilde{A}$.}

\begin{prp}{\it  If $Z$ and $T$ are square integrable, if $\varphi$ and $\chi$ are ${\cal C}^1_b$ 
$$\frac{1}{\varepsilon}\mathbb{E}[(\varphi(Y_\varepsilon)-\varphi(Y)(\chi(Y_\varepsilon)-\chi(Y)]
\rightarrow\mathbb{E}[T^2\varphi^\prime(Y)\chi^\prime(Y)]$$
and in the multidimensional case
$$\frac{1}{\varepsilon}\mathbb{E}[(\varphi(Y_\varepsilon)-\varphi(Y)(\chi(Y_\varepsilon)-\chi(Y)]
\rightarrow\mathbb{E}[(\nabla\varphi)^t(Y)TT^t\nabla\chi(Y)].$$}
\end{prp}
\begin{proof} The demonstration is similar with a first order expansion.\end{proof}
In order to exhibit the operator $\widetilde{A}$, we must examine the conditions of an integration by parts in the preceding
limit. Let us put $\theta_{ij}(y)=\mathbb{E}[(TT^t)_{ij}|Y\!=\!y]$ so that $\mathbb{E}[(\nabla\varphi)^t(Y)TT^t\nabla\chi(Y)]=
\sum_{ij}\mathbb{E}_Y[\varphi^\prime_i\theta_{ij}\chi^\prime_i]$.\\

\begin{prp}{\it If $Z$ and $T$ are square integrable, if for $i,j=1,\ldots,p$ the measure $\theta_{ij}\mathbb{P}_Y$ on $\mathbb{R}^p$
possesses a partial derivative in the sense of distributions $\partial_j(\theta_{ij}\mathbb{P}_Y)$ which is a bounded 
measure absolutely continuous with respect to $\mathbb{P}_Y$, say $\rho_{ij}\mathbb{P}_Y$, then as soon as $\theta_{ij}$ and $\rho_{ij}\in L^2(\mathbb{P}_Y)$ the form $\widetilde{\cal E}[\varphi,\chi]=
\frac{1}{2}\sum_{ij}\mathbb{E}_Y[\varphi^\prime_i\theta_{ij}\chi^\prime_j]$
is closable on the algebra ${\cal D}={\cal C}^2_b$, hypotheses} {\rm(H1)} {\it to} {\rm(H4)} {\it are fulfilled and 
$$\widetilde{A}[\varphi]=\frac{1}{2}\sum_{ij}\theta_{ij}\varphi^{\prime\prime}_{ij}+\frac{1}{2}\sum_{ij}\rho_{ij}\varphi^\prime_j.$$}
\end{prp}
\begin{proof} We have
$$\sum_{ij}\int\theta_{ij}\varphi^\prime_i\chi^\prime_j\,d\mathbb{P}_Y=
\sum_{ij}\int\theta_{ij}(\partial_j(\varphi_i^\prime\chi)-\varphi^{\prime\prime}_{ij}\chi)d\mathbb{P}_Y$$
and the equality
$$\int\theta_{ij}\partial_j(\varphi^\prime_i\chi)d\mathbb{P}_Y=-\int\varphi^\prime_i\chi\rho_{ij}d\mathbb{P}_Y$$
valid for $\varphi,\chi\in{\cal C}^\infty_K$ extends, under the assumptions of the statement, to 
$\varphi,\chi\in{\cal C}^2_b$. This yields 
$\frac{1}{2}\sum_{ij}\mathbb{E}[\varphi^\prime_i\theta_{ij}\chi^\prime_j]=-\frac{1}{2}\int(\sum_{ij}\theta_{ij}\varphi^{\prime\prime}_{ij}
+\sum_{ij}\rho_{ij}\varphi^\prime_j)\chi\,d\mathbb{P}_Y.$
\end{proof}

The operator $\widetilde{A}$ depends only on $T$, not on $Z$. We obtain $\underline{A}$ by difference :
$$\underline{A}[\varphi]= \frac{1}{2}\sum_{ij}\theta_{ij}\varphi^{\prime\prime}_{ij}
+\sum_j(\sum_i\rho_{ij}-z_j)\varphi^\prime_j$$ where $z_j(y)=\mathbb{E}[Z_j|Y\!=\!y]$. At last, $\Abar$ is first order :
$\Abar[\varphi]=\sum_j(z_j-\frac{1}{2}\sum_i\rho_{ij})\varphi^\prime_j.$

For infinite dimensional examples see \cite{JFA}. \\

We are now able to make more precise the dichotomy of \S 1: we shall say that the approximation is {\it weakly stochastic} if hypotheses H1 to H4 are fulfilled and  $\widetilde{A}=0$ and hence $\Abar=
\overline{A}=-\underline{A}$. 

And the approximation will be said {\it strongly stochastic} if hypotheses H1 to H4 are fulfilled and $\widetilde{A}\neq 0$ hence the Dirichlet form (cf. Thm 2.1) is not nought.

\section{The usual norm-based numerical analysis revisited.}
For boundary value problems or optimization problems etc. the resolution by approximation is often displayed in numerical analysis in the following manner : 

The data are represented by a function $f$ in some functional space $F$, the parameters of the problem are represented by a point $\lambda$ in a suitable space $\Lambda$ and the mathematical solution of the problem writes 
$$g=\Phi(f,\lambda).$$ The solution belongs to the space $G$ when $f\in F$ and $\lambda\in\Lambda$. Then, the analysis of the functional $\Phi$ yields norm estimates of the form :
\begin{equation}
\left.
\begin{array}{c}
\|f_n-f\|_F\leq\alpha\\
\|\lambda_n-\lambda\|_\Lambda\leq\beta
\end{array}
\right\}
\Longrightarrow \|g_n-g\|_G\leq\xi(\alpha,\beta, n)
\end{equation}
for some function $\xi$, which assures the convergence of the resolution procedure.

It has to be emphasized that such a reasoning supposes that the premises of (3.1) be fullfilled. The error $f-f_n$ is thought deterministically. The possible randomness of the error and the behaviour of its bias through the functional $\Phi$ are not taken in account in this approach.\\

\noindent{\bf Remark.} a) When the problem is {\it purely mathematical}, the above difficulty may, most often, be considered of secondary importance. Indeed, if the function $f$ and parameters $\lambda$ are random, we may consider that the problem is solved as soon as we are able to compute {\it the law} of the output $g$ or to have an approximation of it (if there were also randomness in the functional $\Phi$, our aim would be to get the joint law of $(f,\lambda, g)$, we move away this case which is similar for simplicity). Now for this, it is enough that an approximation $f_n$ and $\lambda_n$ of $f$ and $\lambda$ yields an approximation $g_n$ of $g$ in probability. In other words, estimates of the form (3.1) in probability are sufficient to solve the problem:
\begin{equation}
\left.
\begin{array}{c}
\forall\delta>0,\;\exists\varepsilon>0, \;\mbox{ s.t.}\hspace{6cm}\\
\mathbb{P}\{\|f_n-f\|_F\leq\alpha; \|\lambda_n-\lambda\|_\Lambda\leq\beta\}\geq1-\varepsilon\hspace{3cm}\\
\hspace{3cm}\Rightarrow\mathbb{P}\{\|g_n-g\|_G\leq\xi(\alpha,\beta, n)\}\geq1-\delta
\end{array}
\right\}
\end{equation}
then the law of $g$ may be approximated by Monte Carlo methods, because we are allowed to choose the sample $f$ and parameter $\lambda$ as we want provided they follow the right probability law.

This can be said otherwise : from a mathematical point of view, most often, the sensitivity of $g$ to the input $f$ may be thought {\it globally}. Estimate like (3.2) will be usually obtained by inequalities similar to (3.1) but in the sense of spaces like $L^p(\Omega,\mathcal{A},\mathbb{P}; F)$, $L^p(\Omega,\mathcal{A},\mathbb{P}; \Lambda)$ and $L^p(\Omega,\mathcal{A},\mathbb{P}; G)$.

b) But different is the situation where $f$ comes from an experiment. For example the temperatures, the wind velocities, etc. in a meteorological model. In such cases, the data $f$ is imposed, known with some precision, and the question whether the errors are weakly or strongly stochastic is relevant. In the first case the sensitivity analysis reduces to a derivation (in a suitable sense between suitable spaces), in the second case a second order Ito-like calculus is compulsory.

The importance of this discussion is reinforced by the results of the next section.
\newpage
\section{The errors due to the graduation of measuring instruments are strongly stochastic.}
Suppose $Y$ is a real quantity measured with a graduated instrument. Let $Y_n$ be the approximation of $Y$ to the nearest graduation, i.e.
$$Y_n=\frac{[nY]}{n}+\frac{1}{2n}$$
($[x]$ denotes the integral part of $x$, and $\{x\}=x-[x]$ the fractional part).

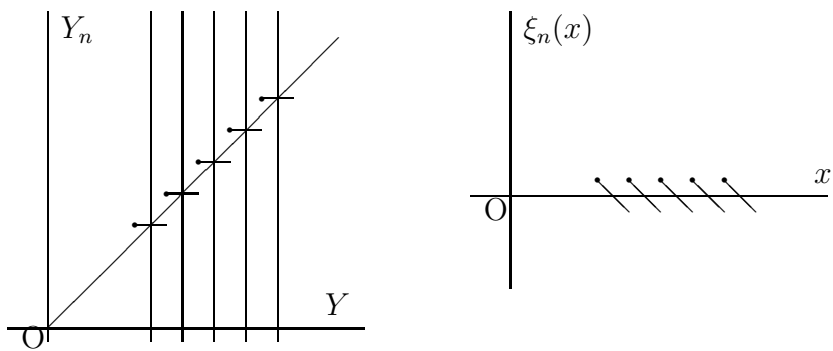
\setlength{\unitlength}{0.5pt}
\begin{picture}(700,300)(35,10)
\put(20,30){\line(1,0){270}}
\put(50,20){\line(0,1){250}}
\put(224,20){\line(0,1){250}}
\put(200,20){\line(0,1){250}}
\put(176,20){\line(0,1){250}}
\put(152,20){\line(0,1){250}}
\put(128,20){\line(0,1){250}}
\put(50,30){\line(1,1){220}}
\put(116,108){\circle*{3}}
\put(116,108){\line(1,0){24}}
\put(140,132){\circle*{3}}
\put(140,132){\line(1,0){24}}
\put(164,156){\circle*{3}}
\put(164,156){\line(1,0){24}}
\put(188,180){\circle*{3}}
\put(188,180){\line(1,0){24}}
\put(212,204){\circle*{3}}
\put(212,204){\line(1,0){24}}
\put(60,250){$Y_n$}
\put(30,15){O}\put(260,40){$Y$}
\put(370,130){\line(1,0){270}}
\put(400,60){\line(0,1){210}}
\put(466,142){\circle*{3}}
\put(466,142){\line(1,-1){24}}
\put(490,142){\circle*{3}}
\put(490,142){\line(1,-1){24}}
\put(514,142){\circle*{3}}
\put(514,142){\line(1,-1){24}}
\put(538,142){\circle*{3}}
\put(538,142){\line(1,-1){24}}
\put(562,142){\circle*{3}}
\put(562,142){\line(1,-1){24}}
\put(410,250){$\xi_n(x)$}
\put(380,112){O}
\put(630,140){$x$}
\end{picture}\\
Let us put $Y_n=Y+\xi_n(Y)$ where the function $\xi_n(x)=\frac{[nx]}{n}-\frac{1}{2n}-x$ is periodic with period $\frac{1}{n}$ and may be written $\xi_n(x) =\frac{1}{n}\theta(nx)$ with $\theta(x)=\frac{1}{2}-\{x\}.$ Let $\mathbb{P}_Y$ be the law of $Y$, we have
\begin{thm}
 a) {\it If $\mathbb{P}_Y$ has a density,
\begin{equation}
(n(Y_n-Y),Y)\quad\stackrel{d}{\Longrightarrow}\quad(V,Y)
\end{equation}
where $V$ is uniform on $(-\frac{1}{2},\frac{1}{2})$ independent of $Y$, and for $\varphi\in\mathcal{C}^1\cap Lip$
\begin{equation}n^2\mathbb{E}[(\varphi(Y_n)-\varphi(Y))^2]\quad\rightarrow\quad\frac{1}{12}\mathbb{E}_Y[\varphi^{\prime 2}].\end{equation}}
b) If $\mathbb{P}_Y$ has a density satisfying one of the following conditions :

i) the derivative in distribution sense $\partial\mathbb{P}_Y$ is a measure $\ll\mathbb{P}_Y$ of the form $\rho\mathbb{P}_Y$ with $\rho\in L^2(\mathbb{P}_Y)$,

ii) $\mathbb{P}_Y=h.1_G\frac{dy}{|G|}$ with $G$ open set, $h\in H^1\cap L^\infty(G)$, $h>0$,\\
\noindent then hypotheses H1 to H4 are fulfilled on the algebra $\mathcal{D}=\mathcal{C}_b^2$ of bounded functions with bounded derivatives up to order 2 with $\alpha_n=n^2$ and 
$$
\begin{array}{rl}
\overline{A}[\varphi]&=\frac{1}{24}\varphi^{\prime\prime}\\
\widetilde{A}[\varphi]&=\frac{1}{24}\varphi^{\prime\prime}+\frac{1}{24}\rho\varphi^\prime\qquad\mbox{ case i)}\\
\widetilde{A}[\varphi]&=\frac{1}{24}\varphi^{\prime\prime}+\frac{1}{24}hh^\prime\varphi^\prime\qquad\mbox{ case ii).}
\end{array}
$$
\end{thm}
\noindent Here $\stackrel{d}{\Longrightarrow}$ denotes the weak convergence, i.e. the convergence of probability measures on continuous bounded functions, $\mathbb{E}_Y$ is the expectation under $\mathbb{P}_Y$.
\begin{proof} a) It is equivalent to study the weak convergence of  $(\frac{1}{2}+n(Y_n-Y),Y)=(\frac{1}{2}+\theta(nY),Y)$. Since $\frac{1}{2}+\theta$ takes its values in the unit interval, it is enough to study the convergence on the characters of the group $\mathbb{T}^1\times\mathbb{R}$, i.e. 
$$\mathbb{E}[e^{2i\pi k(\frac{1}{2}+\theta(nY))}e^{iuY}]=\mathbb{E}[e^{-2i\pi knY}e^{iuY}]=\Psi_Y(u-2\pi kn)$$
where $\Psi_Y$ is the characteristic function of $Y$. This tends to $\Psi(u)1_{\{k\neq0\}}$ by the Riemann-Lebesgue lemma since  $ \mathbb{P}_Y$ has a density.

If $\varphi\in\mathcal{C}^1\cap Lip$, the relation $\varphi(Y_n)-\varphi(Y)=(Y_n-Y)\int_0^1\varphi^\prime(Y+\alpha(Y_n-Y))d\alpha$ gives
$$n^2\mathbb{E}[(\varphi(Y_n)-\varphi(Y))^2]=\mathbb{E}[\theta^2(nY)\varphi^{\prime 2}(Y)]+o(1)$$ and $\mathbb{E}[\theta^2(nY)\varphi^{\prime 2}(Y)]\rightarrow\int_{-\frac{1}{2}}^{\frac{1}{2}}\theta^2(t)dt\mathbb{E}[\varphi^{\prime 2}(Y)]$ what proves the second assertion.

b) We postpone the proof of b) after theorem 4.2.
\end{proof}
\noindent{\bf Comments.} 1) The result a) is the classical {\it arbitrary functions principle} (cf. \cite{Poin} \cite{Hopf}), it would be still valid if $\mathbb{P}_Y$ were a {\it Rajchman measure} (see \cite{IHP}). For extensions of the arbitrary functions principle to infinite dimensional cases see \cite{CRAS} and \cite{IHP}. A summary of the history of this principle is given in \cite{IHP} section I.3.

2) The b) of the theorem shows that when the law $\mathbb{P}_Y$ is smooth, the approximation $Y_n$ of $Y$ to the nearest graduation  is {\it strongly stochastic}. \\

The results of theorem 4.1 extend to the finite dimensional case: Let us suppose $Y$ is $\mathbb{R}^d$-valued, measured with an equidistant graduation corresponding to an orthonormal rectilinear coordinate system, and estimated to the nearest graduation component by component. Thus we put
$$Y_n=Y+\frac{1}{n}\theta(nY)$$ with $\theta(y)=(\frac{1}{2}-\{y_1\},\cdots,\frac{1}{2}-\{y_d\})$.\\

\begin{thm} {\it a) If $\mathbb{P}_Y$ has a density and if $X$ is $\mathbb{R}^m$-valued

\begin{equation}(X,n(Y_n-Y))\quad\stackrel{d}{\Longrightarrow}\quad(X,(V_1,\ldots,V_d))\end{equation}
where the $V_i$'s are i.i.d. uniformly distributed on $(-\frac{1}{2},\frac{1}{2})$ independent of $X$.

\noindent For all $\varphi\in\mathcal{C}^1\cap Lip(\mathbb{R}^d)$
\begin{equation}
(X,n(\varphi(Y_n)-\varphi(Y)))\quad\stackrel{d}{\Longrightarrow}\quad(X,\sum_{i=1}^dV_i\varphi^\prime_i(Y))
\end{equation}
\begin{equation}
n^2\mathbb{E}[(\varphi(Y_n)-\varphi(Y))^2Ü|Y\!=\!y]\rightarrow\frac{1}{12}\sum_{i=1}^d\varphi^{\prime 2}_i(y)\qquad\mbox{ in }L^1(\mathbb{P}_Y)
\end{equation}
in particular

\begin{equation}
n^2\mathbb{E}[(\varphi(Y_n)-\varphi(Y))^2]\rightarrow\mathbb{E}_Y[\frac{1}{12}\sum_{i=1}^d\varphi^{\prime 2}_i(y)].
\end{equation}
\indent b) If $\varphi$ is of class $\mathcal{C}^2$, the conditional expectation $n^2\mathbb{E}[\varphi(Y_n)-\varphi(Y)|Y=y]$ possesses a version $n^2(\varphi(y+\frac{1}{n}\theta(ny))-\varphi(y))$ independent of the probability measure $\mathbb{P}$ which converges in the sense of distributions to the function $\frac{1}{24}\bigtriangleup\varphi$.

c) If $\mathbb{P}_Y\ll dy $ on $\mathbb{R}^d$, $\forall\psi\in L^1([0,1])$
\begin{equation}(X,\psi(n(Y_n-Y)))\quad\stackrel{d}{\Longrightarrow}(X,\psi(V)).\end{equation}

d) We consider the bias operators on the algebra $\mathcal{C}^2_b$ of bounded functions with bounded derivatives up to order 2 with the sequence $\alpha_n=n^2$. If $\mathbb{P}_Y$ has a density and if one of the following condition is fulfilled

i) $\forall i=1,\ldots,d$ the partial derivative $\partial_i\mathbb{P}_Y$ in the sense of distributions is a measure $\ll\mathbb{P}_Y$ of the form $\rho_i\mathbb{P}_Y$ with $\rho_i\in L^2(\mathbb{P}_Y)$

ii) $\mathbb{P}_Y=h1_{G}\frac{dy}{|G|}$  with $G$ open set, $h\in H^1\cap L^\infty(G)$, $h>0$

\noindent then hypotheses {\rm (H1)} to {\rm (H4)} are satisfied and 
$$\begin{array}{rl}
\overline{A}[\varphi]&=\frac{1}{24}\bigtriangleup\varphi\\
\widetilde{A}[\varphi]&=\frac{1}{24}\bigtriangleup\varphi+\frac{1}{24}\sum\varphi^\prime_i\rho_i\qquad\mbox {case i)}\\
\widetilde{A}[\varphi]&=\frac{1}{24}\bigtriangleup\varphi+\frac{1}{24}\frac{1}{h}\sum h^\prime_i\varphi^\prime_i\qquad\mbox {case ii)}\\
\Gamma[\varphi]&=\frac{1}{12}\sum \varphi^{\prime 2}_i.
\end{array}
$$}
\end{thm}

\begin{proof} The argument for relation (4.3) is similar to one dimensional case. The relation (4.4) comes from the Taylor expansion 
$\varphi(Y_n)-\varphi(Y)=$

$=\sum_{i=1}^d (Y_{n,i}-Y_i)\int_0^1\varphi^\prime_i(Y_{n,1},\ldots,Y_{n,i-1},Y_i+t(Y_{n,i}-Y_i),Y_{i+1},\ldots,Y_d)\,dt$

\noindent and the convergence
$$(X,\sum_i\theta(nY_i)\varphi^\prime_i(Y))\quad\stackrel{d}{\Longrightarrow}\quad(X,\sum_i\varphi^\prime_i(Y)V_i)$$
thanks to (4.3) and the following approximation in $L^1$ 
$$\mathbb{E}\left|\sum_i\theta(nY_i)\varphi^\prime_i(Y)-\sum_i\theta(nY_i)\int_0^1\varphi^\prime_i(\ldots,Y_i+t(Y_{n,i}-Y_i),\ldots)dt\right|\rightarrow0.$$
To prove the formulae (4.5) and (4.6) let us remark that

\begin{flushleft}
$n^2\mathbb{E}[(\varphi(Y_n)-\varphi(Y)^2|Y=y]=$
\end{flushleft}
$$=\mathbb{E}\left[\left|\sum_i\theta(nY_i)\int_0^1\varphi^\prime_i(\ldots,Y_i+t(Y_{n,i}-Y_i),\ldots)dt\right|^2|Y=y\right]$$
$$=\left|\sum_i\theta(ny_i)\int_0^1\varphi^\prime_i(y_1+\frac{1}{n}\theta(ny_1),\ldots,y_i+t\frac{1}{n}\theta(ny_i),\ldots)dt\right|^2\quad\mathbb{P}_Y-a.s.$$
each term $(\theta(ny_i)\int_0^1\varphi^\prime_i(\ldots)dt)^2$ converges to $\int\theta^2\varphi^{\prime 2}_i(y)=\frac{1}{12}\varphi^{\prime 2}_i$ in $L^1$ and each term $\theta(ny_i)\theta(ny_j)\int_0^1\ldots\int_0^1\ldots$ goes to zero in $L^1$ what proves the a) of the statement.

The point b) is obtained following the same lines with a Taylor expansion up to second order and an integration by part thanks to the fact that $\varphi$ is now supposed to be $\mathcal{C}^2$.

In order to prove c) let us suppose first that $\mathbb{P}_Y=1_{[0,1]^d}.dy$. Considering a sequence of functions $\psi_k\in\mathcal{C}_b$ tending to $\psi$ in $L^1$ we have the bound 
$$
\begin{array}{l}
|\mathbb{E}[e^{i<u,X>}e^{iv\psi(\theta(nY))}]-\mathbb{E}[e^{i<u,X>}e^{iv\psi_k(\theta(nY))}]|\\
\leq |v|\int|\psi(\theta(ny))-\psi_k(\theta(ny))|dy\\
=|v|\sum_{p_1=0}^{n-1}\cdots\int_{p_1}^{p_1+1}\cdots|\psi(\theta(ny_1)\ldots)-\psi_k(\theta(ny_1)\ldots)|dy_1\ldots dy_d\\
=|v|\sum\cdots\sum\int\cdots\int|\psi(\theta(x_1),\ldots)-
\psi_k(\theta(x_1),\ldots)|\frac{dx_1}{n}\cdots\frac{dx_d}{n}\\
=|v|\|\psi-\psi_k\|_{L^1}.
\end{array}
$$
And this yields (4.7) in this case. Now if $\mathbb{P}_Y\ll dy$ then $\mathbb{P}_{\{Y\}}\ll dy $ on $[0,1]^d$ and the weak convergence under $dy$ on $[0,1]^d$ implies the weak convergence under $\mathbb{P}_{\{Y\}}$ what yields the result.

In d) the point i) is proved by the approach already used in \cite{JFA} consisting of proving  that hypothesis (H3) is fulfilled by displaying the operator $\widetilde{A}$ thanks to an integration by parts. The point ii) is an application of a Girsanov theorem for Dirichlet forms (cf. \cite{IHP}).\end{proof}

\section{Conclusion.}
The question of the {\it specification} of an approximate numerical result may be made more precise : it is a description of the error on the inputs in such a way that it is possible (in smooth cases) to obtain the same kind of description on the output.

$\bullet$ To give the result with an interval for the error is a specification. But it is unsatisfactory for several reasons :

(i) the law of the error may have a non compact support, or a support not decreasing to a point (cf. Polya's urn),

(ii) with such a description we can manage neither the variances nor the biases.

$\bullet$ To give the result with an interval and a probability that the error be inside this interval is also a specification.  It is a triplet
 $(x_n, \alpha, \varepsilon)$ where $x_n$ is the proposed result, with the condition
 $$\mathbb{P}\{|x-x_n|<\alpha\}\geq 1-\varepsilon.$$
 As already discussed, such a specification may be used when we are only concerned by the law of the output. If the probability 
$\mathbb{P}\{|x-x_n|<\alpha\}$ is known for every $\alpha$, this gives the knowledge of $\|x-x_n\|_{L^2(\mathbb{P})}^2$, but the critique (ii) still holds.

$\bullet$ The Dirichlet theoretical specification used in our argumentation deals with the following mathematical objects :

\hspace{.5cm}. $\mathbb{P}_Y$ the law of the quantity to be approximated,

\hspace{.5cm}. the sequence $\alpha_n$ giving the order of magnitude,

\hspace{.5cm}. the algebra $\mathcal{D}$,

\hspace{.5cm}. $\overline{A}$, $\underline{A}$ the theoretical and practical bias operators, 

\hspace{.5cm}. $\Gamma$ the square field operator of the associated Dirichlet form.

This specification seems to be too sophisticated to be used by engineers in usual cases, and the question remains to simplify it, preserving the main ideas.

Here we will just give a comment on this question in the finite dimensional case $Y=\Phi(X)$ with $\Phi$ regular from $\mathbb{R}^p$ into $\mathbb{R}^q$. If the input is measured with a graduated instrument, the square field operator on the input $\Gamma_{in}$ is yielded by the size of the graduation and do not depend on the probability law of the input provided that this law be regular, by the arbitrary functions principle. A natural hypothesis is to suppose that the law of the input is uniform in a neighbourhood of the numerical data. Then (theorem 4.2 d)) the approximation of the input satisfies
$$\overline{A}_{in}=\underline{A}_{in}=\widetilde{A}_{in}\qquad(=\frac{1}{12}\Delta)$$ and this equality will be transported to the output
$$\overline{A}_{out}=\underline{A}_{out}=\widetilde{A}_{out}$$ (see definitions H1 to H3). The generator $\widetilde{A}_{out}$ and the square field operator $\Gamma_{out}$ will be given by {\it the image} of the input Dirichlet structure by the map $\Phi$ (cf. \cite{BH} Chap V, \cite{NBlivre} Chap IV). The formulae are 
\begin{equation}
\left.
\begin{array}{rl}
(\Gamma_{out}[u])(y)&=\mathbb{E}[\Gamma_{in}[u\circ\Phi](X)|Y=y]\\
(\widetilde{A}_{out}[u])(y)&=\mathbb{E}[\widetilde{A}_{in}[u\circ\Phi](X)|Y=y]
\end{array}\right\}
\end{equation}
where $\Gamma_{in}[u\circ\Phi]$ and $\widetilde{A}_{in}[u\circ\Phi]$ are obtained by the functional calculus in Dirichlet structures (cf. \cite{BH} Chap I section 6) with natural notation this writes
\begin{equation}
\left.
\begin{array}{rl}
\Gamma_{in}[u\circ\Phi]&=(\nabla u)^t\!\circ\Phi\;\,\Gamma_{in}[\Phi]\;(\nabla u)\!\circ\Phi\\
\widetilde{A}_{in}[u\!\circ\Phi]&=(\nabla u)^t\!\circ\Phi\;\,\widetilde{A}_{in}[\Phi]+\frac{1}{2}\sum_{ij}\partial_{ij}^2u\;\Gamma_{in}[\Phi_i,\Phi_j]
\end{array}\right\}
\end{equation}
We see that, in order to obtain the coefficients of the bias differential operator $\widetilde{A}_{out}$, by formulae (5.2) we have to compute $\Gamma_{in}[\Phi_i,\Phi_j]$ and $\widetilde{A}_{in}[\Phi]$ which involves the {\it Jacobian} and the {\it Hessian} matrices of the map $\Phi$ and then by formulae (5.1) to average in $X$ on the level manifolds of $\Phi$.

In conclusion, we have attempted to convince the reader that errors have to be thought in terms of second order differential operators. In order that this language be convenient for practical engineering use, more simplicity has to be looked for, taking in account the specific form of the different problems.


\end{document}